\documentclass{amsart}

\usepackage{amssymb}

\usepackage[dvips]{graphicx}

\usepackage{color}

\theoremstyle{plain}

\newtheorem{mainthm}{Theorem}

\newtheorem*{conj*}{Conjecture}
\newtheorem*{cor*}{Corollary}
\newtheorem*{def*}{Definition}

\newtheorem{theorem}{Theorem}[section]
\newtheorem{proposition}{Proposition}

\newtheorem{lemma}[theorem]{Lemma}

\newtheorem{definition}{Definition}

\newcommand{\Z}{\mathbb{Z}}
\newcommand{\N}{\mathbb{N}}

\newcommand{\eps}{\varepsilon}

\title{Positively $N$-expansive homeomorphisms and the L-shadowing property.}

\author[Bernardo Carvalho and Welington Cordeiro]{}

\begin{document}

\maketitle
\centerline{\scshape Bernardo Carvalho, Welington Cordeiro}
{\footnotesize
} 
\begin{abstract}{We discuss further the dynamics of $n$-expansive homeomorphisms with the shadowing property, started in \cite{CC}. The L-shadowing property is defined and the dynamics of $n$-expansive homeomorphisms with such property is explored. In particular, we prove that positively $n$-expansive homeomorphisms with the L-shadowing property can only be defined in finite metric spaces.}
\end{abstract}



\bigskip


\section{Introduction and statement of results}

Expansiveness and shadowing properties have been discussed since the beginnings of dynamical systems theory. They are so important to the hyperbolic theory that homeomorphisms admitting them are usually called \emph{topologically hyperbolic}, since their dynamics and the hyperbolic dynamics are pretty much the same. Many authors considered generalizations of both shadowing and expansiveness (see \cite{DH}, \cite{FG}, \cite{KT}, \cite{LZ}, \cite{Mor} among others). One of them, which is receiving much attention recently, is \emph{$n$-expansiveness}. These systems were introduced in \cite{Mor}, where for each integer $n>1$, non-trivial examples of \emph{positively $n$-expansive homeomorphisms} were given. 

\begin{definition}\label{23}
We say that a continuous map $f$, defined in a metric space $(X,d)$, is positively $n$-expansive if there exists $c>0$ such that for each $x\in X$ the set $$W^s_{c}(x):=\{y\in X; \,\, d(f^k(y),f^k(x))\leq c \,\,\,\, \textrm{for every} \,\,\,\, k\geq 0\}$$ contains at most $n$ different points of $X$. The number $c$ is called the positively $n$-expansive constant of $f$ and the set $W^s_c(x)$ is the local stable set of $x$ of size $c$.
\end{definition}

Positively $1$-expansive homeomorphisms are exactly the positively expansive homeomorphisms discussed in \cite{KR}. It is proved there that if a compact metric space $X$ admits a positively expansive homeomorphism, then $X$ is finite. Despite this statement does not contain any shadowing property in it, we note that there is a hidden dynamical property which proves the space is finite: the \emph{L-shadowing property}.

 \begin{definition}
We say that a homeomorphism, defined in a compact metric space $X$, has the L-shadowing property, if for every $\eps>0$, there exists $\delta>0$, such that for every $\delta$-pseudo orbit $(x_k)_{k\in\Z}$, that is also a two-sided limit pseudo-orbit, there is $z\in X$ that $\eps$-shadows and also two-sided limit shadows $(x_k)_{k\in\Z}$.
\end{definition}

The reader which is not familiar with the notions of $\delta$-pseudo orbit, two-sided limit pseudo orbit, $\eps$-shadow and two-sided limit shadow is invited to read precise definitions in Section \ref{positiv}. The L-shadowing property is not so different from other shadowing properties found in the literature (see \cite {C}, \cite{GOP}, \cite{O} and \cite{P1}, for example). The difference, however, is the following: while the shadowing property assures the existence of one point that $\eps$-shadows $(x_k)_{k\in\Z}$ and the two-sided limit shadowing property assures the existence of another one that two-sided limit shadows $(x_k)_{k\in\Z}$, the L-shadowing property obtains one point that both $\eps$-shadow and two-sided limit shadow it. This difference spreads to the dynamical consequences of this property: it implies both shadowing and a finite number of chain recurrent classes (see Proposition \ref{3}), while examples of homeomorphisms with the shadowing property and an infinite number of chain recurrent classes can be found in \cite{CC} and the two-sided limit shadowing property implies even topological mixing. 

We prove that this property is present in topologically hyperbolic homeomorphisms (see Proposition \ref{expansive}) and obtain that the shadowing property and the L-shadowing property are equivalent for expansive homeomorphisms. However, for each $n>1$, examples of $n$-expansive homeomorphisms with the shadowing property but without the L-shadowing property can be found in \cite{CC}. Indeed, the proof of Proposition \ref{expansive} does not follow immediately to the $n$-expansive scenario (see the brief discussion after its proof in Section \ref{localtsls}). In this paper, we prove that the L-shadowing property is equivalent to the space being finite, for positively $n$-expansive homeomorphisms.

\begin{mainthm}\label{h} If a positively $n$-expansive homeomorphism is defined in a compact metric space $X$ and has the L-shadowing property, then $X$ is finite. 
\end{mainthm}

This theorem generalizes the main theorem of \cite{KR} for positively expansive homeomorphisms. Indeed, it is proved in \cite{Sakai} that positively expansive maps have the shadowing property if, and only if, they are open maps, proving, in particular, that any positively expansive homeomorphism admits the L-shadowing property, in view of Proposition \ref{expansive}. We further prove that the metric space is finite assuming transitivity and the shadowing property.

\begin{mainthm}\label{bla}
If a positively $n$-expansive homeomorphism is defined in a compact metric space $X$, is transitive and admits the shadowing property, then $X$ is finite.
\end{mainthm}

As a consequence, the examples given at \cite{Mor} do not admit the shadowing property. This also proves the known fact that the minimal set of the Denjoy map does not have the shadowing property (see \cite{P} for the original proof). In \cite{APV}, Artigue, Pac\'ifico and Vieitez proved that if a transitive $2$-expansive homeomorphisms is defined in a compact surface, then it is expansive. They also give examples of $2$-exansive homeomorphisms, defined on surfaces, that are not expansive. So the transitivity assumption is necessary to these results. Theorem \ref{bla} goes in a similar direction, proving that transitive positively $n$-expansive homeomorphisms are actually positively expansive, assuming the shadowing property.




One important difference between the expansive homeomorphisms and the $n$-expansive ones is that for expansive homeomorphisms there exists $\eps>0$ such that the local stable set of $x$ of size $\eps$ is contained in the stable set of $x$ ($W^s(x)$) for every $x\in X$ (see \cite{Ma} for a proof), while for $n$-expansive homeomorphisms such number does not exist, even when the shadowing property is present (see the examples in \cite{CC}). This happens due to the existence of points in different stable sets belonging to the same local stable set. So the authors introduced the \emph{number of different stable sets in the local stable set of $x$ of size $\eps$}, denoted by $n(x,\eps)$ (see the precise definition in Section 2). The number of different unstable sets in the local unstable set of $x$ of size $\eps$ ($W^u_{\eps}(x)$) was introduced in a similar way and was denoted by $\bar{n}(x,\eps)$. It is clear that $W^s_{\eps}(x)\subset W^s(x)$ if, and only if, $n(x,\eps)=1$, and that $W^u_{\eps}(x)\subset W^u(x)$ if, and only if, $\bar{n}(x,\eps)=1$. The numbers $n(x,\eps)$ and $\bar{n}(x,\eps)$ can be greater than one. Indeed, in the examples of \cite{CC}, one can find, for every $\eps>0$, infinite points $x\in X$ satisfying $n(x,\eps)>1$. However, it is proved, in \cite{CC} Proposition 4.4, that if an $n$-expansive homeomorphisms has the shadowing property, then there exists $\eps>0$ such that both numbers $n(x,\eps)$ and $\bar{n}(x,\eps)$ does not exceed $n$ for every $x\in X$. When the L-shadowing property is present we obtain stronger consequences.  





\begin{mainthm}\label{D}
If an $n$-expansive homeomorphism $f$, defined in a compact metric space, has the L-shadowing property, then there exists $\eps>0$ such that the product $\bar{n}(x,\eps).n(x,\eps)$ does not exceed $n$ for every $x\in X$. If, in addition, $f$ is transitive, then $\bar{n}(x,\eps).n(y,\eps)\leq n$ for every $x,y\in X$.
\end{mainthm}

We remark that for non-transitive homeomorphisms, the product $\bar{n}(x,\eps).n(y,\eps)$ does not exceed $n$ either, provided $x$ and $y$ are suficiently close. The paper is organized as follows: in Section 2 we state all necessary definitions and prove Theorem \ref{bla}, while in Section 3 we obtain the main consequences of the L-shadowing property and prove Theorems \ref{h} and \ref{D}.

\section{Positively $N$-expansive homeomorphisms}\label{positiv}

Let $f$ be a homeomorphism defined in a compact metric space $(X,d)$. We denote by $B(x,\eps)$ the ball centered at $x$ and radius $\eps$ in $X$, which is the set of points whose distance to $x$ is smaller than $\eps$. The supremmum of the distance between points of $X$ is called the \emph{diameter of the space}. The \emph{orbit} of a point $x\in X$ is the set $\displaystyle{\{f^k(x)\,;\,\, k\in\Z\}}$. When this set is finite, we say that $x$ is a \emph{periodic point}. In this case, there exists $k\in\N$ such that $f^k(x)=x$. The set of all periodic points will be denoted by $Per(f)$. 

The \emph{omega limit set} of a point $x\in X$ is the set $\omega(x)$ of all accumulation points of the future orbit of $x$. A point $x\in X$ is a \emph{non-wandering point} if for each open subset $U$ of $X$ with $x\in X$, there is $k>0$ such that $f^k(U)\cap U\neq\emptyset$. The set of all non-wadering points of $f$ is called the \emph{non-wandering set} and is denoted by $\Omega(f)$. Let $I\subset\Z$ be a nonempty set of consecutive integers. We say that a sequence $(x_k)_{k\in I}\subset X$ is an \emph{$\eps$-pseudo-orbit} if it satisfies $d(f(x_k),x_{k+1})<\eps$ for all $k$ such that $k,k+1 \in I$. A point $x\in X$ is \emph{chain-recurrent} if for each $\eps>0$ there exists a non-trivial finite $\eps$-pseudo orbit starting and ending at $x$. The set of all chain-recurrent points is called the \textit{chain recurrent set} and is denoted by $CR(f)$. This set can be split into disjoint, compact and invariant subsets, called the \emph{chain-recurrent classes}. The \emph{chain-recurrent class} of a point $x\in X$ is the set of all points $y\in X$ such that for every $\eps>0$ there exist a periodic $\eps$-pseudo orbit containing both $x$ and $y$. We say that $f$ is \emph{transitive} if for any pair $(U,V)$ of non-empty open subsets of $X$, there exists $k\in\N$ such that $f^k(U)\cap V\neq\emptyset$. It is easy to see that transitive homeomorphisms admit only one chain recurrent class, that is the whole space. 

Its well known that the closure of the set of all periodic points and also the omega limit set of any point are contained in the non-wandering set, that is, in turn, contained in the chain recurrent set (see, for example, chapter $1$ of \cite{S2}). To define $n$-expansiveness, we consider the \emph{local unstable set of $x$ of size $c$} $$W^u_{c}(x):=\{y\in X; \,\, d(f^k(y),f^k(x))\leq c \,\,\,\, \textrm{for every} \,\,\,\, k\leq 0\}.$$ We say that $f$ is \emph{$n$-expansive} if there exists $c>0$ such that for each $x\in X$ the set $$W^s_{c}(x)\cap W^u_{c}(x)$$ contains at most $n$ different points of $X$. The number $c$ is called the \emph{$n$-expansivity constant of $f$} and the set $W^s_{c}(x)\cap W^u_{c}(x)$ is called the \emph{dynamical ball of x of size $c$}. Moreover, the \emph{stable set of $x$} is the set $$W^s(x):=\{y\in X; \,\, d(f^k(y),f^k(x))\to0,  \,\,\,\, k\to\infty\}$$ while the \emph{unstable set of $x$} is the set $$W^u(x):=\{y\in X; \,\, d(f^k(y),f^k(x))\to0, \,\,\,\,  k\to-\infty\}.$$ For an $n$-expansive homeomorphism $f$, defined in a compact metric space $X$, some $\eps>0$ and $x\in X$, we define the \emph{number of different stable sets} of $f$ in $W^s_{\eps}(x)$ as the number $n(x,\eps)$ satisfying:
\begin{enumerate}
  \item there exists a set $E(x,\eps)\subset W_{\eps}^s(x)$ with $n(x,\eps)$ elements such that if two different points $y,z\in E(x,\eps)$ then $y\notin W^s(z)$,
  \item if $y_1,y_2,\dots,y_{n(x,\eps)+1}$ are $n(x,\eps)+1$ different points of $W_{\eps}^s(x)$, then there exist two different points $y_i,y_j\in\{y_1,y_2,\dots,y_{n(x,\eps)+1}\}$ such that $y_i\in W^s(y_j)$.
\end{enumerate}

More information about this number can be found in \cite{CC}. We will also need several notions of the shadowing property, so we define them precisely. We begin with the (standard) shadowing property. We say that a sequence $(x_k)_{k\in\Z}\subset X$ is $\eps$-\emph{shadowed} if there exists $y\in X$ satisfying $$d(f^k(y),x_k)<\eps, \,\,\,\,\,\, k\in\Z.$$ We say that $f$ has the \emph{shadowing property} if for every $\eps>0$ there exists $\delta>0$ such that every $\delta$-pseudo-orbit is $\eps$-shadowed. This property was extensively studied due to its relation to the hyperbolic and stability theories (see monographs \cite{AH}, \cite{P}). A similar property was introduced by T. Eirola, O. Nevanlinna and S. Pilyugin in \cite{ENP}. It is called \emph{limit shadowing}.
We say that a sequence $(x_k)_{k\in\N}\subset X$ is a \emph{limit pseudo-orbit} if it satisfies $$d(f(x_k),x_{k+1})\rightarrow 0, \,\,\,\,\,\, k\rightarrow\infty.$$ A sequence $\{x_k\}_{k\in\N}$ is \emph{limit-shadowed} if there exists $y\in X$ such that $$d(f^k(y),x_k)\rightarrow 0, \,\,\,\,\,\, k\rightarrow\infty.$$ We say that $f$ has the \emph{limit shadowing property} if every limit pseudo-orbit is limit-shadowed. If we consider bilateral sequences of $X$ we can define the \emph{two-sided limit shadowing property}. We say that $(x_k)_{k\in\Z}$ is a \emph{two-sided limit pseudo-orbit} if it satisfies $$d(f(x_k),x_{k+1})\rightarrow 0, \,\,\,\,\,\, |k|\rightarrow\infty.$$
A sequence $(x_k)_{k\in\Z}$ is \emph{two-sided limit shadowed} if there exists $y\in X$ satisfying $$d(f^k(y),x_k)\rightarrow 0, \,\,\,\,\,\,
|k|\rightarrow \infty.$$ We say that $f$ has the \emph{two-sided limit shadowing property} if every two-sided limit pseudo-orbit is two-sided limit shadowed. A slightly different version was defined by the first author and D. Kwietniak in \cite{CK}, introducing gaps in the shadowing orbits. It is called \emph{two-sided limit shadowing with a gap}.

\begin{definition}
We say that a sequence $(x_k)_{k\in\Z}\subset X$ is \emph{two-sided limit shadowed with gap $K\in\Z$} if there exists $y\in X$ satisfying
\begin{align*}
d(f^k(y),x_k)\to 0, \,\,\,\,\,\, k\to-\infty,\\
d(f^{K+k}(y),x_k)\to 0 \,\,\,\,\,\, k\to\infty.
\end{align*}
For $N\in\N$ we say that $f$ has the \emph{two-sided limit shadowing property with gap $N$} if every two-sided limit pseudo-orbit of $f$ is two sided limit shadowed with gap $K\in \Z$ with $|K|\le N$. We also say that $f$ has the \emph{two-sided limit shadowing property with a gap} if such an $N\in\N$ exists.
\end{definition}

In \cite{C} it is proved that any expansive homeomorphism with the shadowing property has the limit shadowing property and also the two-sided limit shadowing property, provided it is topologically mixing. In \cite{CK} there are examples of transitive expansive homeomorphisms with the shadowing property but without the two-sided limit shadowing property. However, it is proved that these systems admit the two-sided limit shadowing property with a gap. Part of these results was generalized in \cite{CC} for $n$-expansive homeomorphisms with the shadowing property. Precisely, it is proved that $n$-expansive homeomorphisms with the shadowing property have the limit shadowing property and also the two-sided limit shadowing property, if they are topologically mixing. The first result of this section obtains the two-sided limit shadowing property with a gap in the transitive scenario.

\begin{theorem}\label{gap}
If a transitive $n$-expansive homeomorphism has the shadowing property then it has the two-sided limit shadowing property with a gap.
\end{theorem}

To prove it, we basically follow the proof of the expansive case in \cite{CK}. In this scenario, the spectral decomposition theorem (\cite[Theorem 3.1.11]{AH}) assures the existence of an integer $N$ such that $X$ can be written as a disjoint union, $X=B_1\cup\ldots\cup B_N$ of non-empty closed sets satisfying
$$f(B_{i-1})=B_{i\bmod N} \quad\text{for }i=2,\ldots,N,$$ and such that $f^N|_{B_i}\colon B_i \to B_i$ is topologically mixing for each $i$. This decomposition and the fact that $f^N|_{B_i}$ has the two-sided limit shadowing property for every $i$, are enough to prove that $f$ has the two-sided limit shadowing property with gap $N$. We note that the expansiveness is not, indeed, necessary to obtain this decomposition: it is proved in \cite[Theorem 8]{O} that any transitive homeomorphism admitting the shadowing property also admits it. Since $f^N|_{B_i}$ has the two-sided limit shadowing property in the $n$-expansive case (Theorem C in \cite{CC}) we can just repeat the proof of \cite{CK} and prove Theorem \ref{gap}. Another fact that is still true for $n$-expansive homeomorphisms is the following:     

\begin{proposition}\label{2}
If an $n$-expansive homeomorphism, defined in a compact metric space, has the shadowing property, then its chain recurrent set is equal to the closure of the set of its periodic points, i.e., $CR(f)=\overline{Per(f)}$. 
\end{proposition}

\begin{proof} It is enough to prove that each chain-recurrent point is approximated by periodic points. Let $x\in X$ be a chain recurrent point and $\eps>0$ be given. We can assume $\eps$ is smaller than the $n$-expansiviness constant of $f$. Let $0<\delta<\frac{\eps}{2}$, given by the shadowing property, be such that each $\delta$-pseudo orbit is $\frac{\eps}{2}$ shadowed. Since $x$ is chain recurrrent, there exists a periodic $\delta$ pseudo-orbit $\{x_k\}_{k=0}^l\subset X$ such that $x_0=x=x_l$. The shadowing property assures the existence of $z\in X$ satisfying $$d(f^k(z),x_k)<\frac{\eps}{2}, \,\,\,\,\,\, k\in\Z.$$ In particular, $f^{kl}(z)\in W^s_\eps(z)\cap W^u_\eps(z)$ for each $k\in\mathbb{Z}$. Since $\eps$ is smaller than the $n$-expansive constant of $f$, it follows that the orbit of $z$ is finite and $z$ is a periodic point satisfying $d(x,z)<\eps$. Since this can be done for each $\eps>0$, the proposition is proved.
\end{proof}

We prove that finiteness of the non-wandering set is enough to obtain finiteness of the space for positively $n$-expansive homeomorphisms.

\begin{lemma}\label{Lema} If a positively $n$-expansive homeomorphism is defined in a compact metric space $X$, then $\Omega(f)$ is finite if, and only if, $X$ is finite. 
\end{lemma}
\begin{proof} It is obvious that $X$ being finite imples $\Omega(f)$ is finite. Now, finiteness of the non-wandering set implies that all non-wandering points are periodic. We will prove that all points of $X$ are periodic. This is enough to prove that $X$ is finite, since $X$ would be contained in the non-wandering set. Assume $x\in X$ is not a periodic point. We note that $\omega(x)$ is not a single fixed point, since this would imply the existence of infinite different points in the same local stable set, contradicting positively $n$-expansiveness. Then there exists at least two periodic points $p$ and $q$ in the $\omega$-limit set of $x$. We can suppose $p$ and $q$ are fixed points, considering some iterate of $f$. Let $d=d(p,q)$ and $n_1\in\mathbb{N}$ be such that $$\frac{1}{m}<d/4, \,\,\,\,\,\, m\geq n_1.$$ Let $\eps_1=\frac{1}{n_1}$ and choose $\delta_1>0$ such that $d(a,p)<\delta_1$ implies $d(f(a),p)<\eps_1$. Since $p\in \omega(x)$, there exists $k_0\in\N$ such that $f^{k_0}(x)\in B(p,\delta_1)$. Since $q\in\omega(x)$, there exists $k_1$, the least natural number greater than $k_0$, such that $d(f^{k_1}(x),p)>\eps_1$. Observe that $$f^{k_1-1}(x)\in B(p,\eps_1)\setminus B(p,\delta_1).$$ Indeed, $f^{k_1-1}(x)\in B(p,\eps_1)$, since $k_1$ is the least natural number greater than $k_0$ satisfying $d(f^{k_1}(x),p)>\eps_1$, and $f^{k_1-1}(x)\notin B(p,\delta_1)$, because $f^{k_1-1}(x)\in B(p,\delta_1)$ would imply $f^{k_1}(x)\in B(p,\eps_1)$ and $d(f^{k_1}(x),p)<\eps_1$, which is a contradiction. Since $p$ and $q$ are on the $\omega$-limit set of $x$ we can repeat this argument infinite times in the orbit of $x$ and obtain an increasing sequence $(k_m)_{m\in\N}$ of natural numbers such that $$f^{k_m}(x)\in B(p,\eps_1)\setminus B(p,\delta_1)$$ for each $m\in\mathbb{N}$. Hence, there is at least one point $$z_1\in \overline{B(p,\eps_1)\setminus B(p,\delta_1)}\cap\omega(x).$$ Now consider $n_2>n_1$ such that $$\frac{1}{m}<\delta_1, \,\,\,\,\,\, m\geq n_2.$$ Let $\eps_2=\frac{1}{n_2}$ and choose $\delta_2>0$ such that $d(a,p)<\delta_2$ implies $d(f(a),p)<\eps_2$. The same argument assures the existence of $$z_2\in\overline{B(p,\eps_2)-B(p,\delta_2)}\cap\omega(x).$$
It is clear that $z_2\neq z_1$. With an induction process we can find a sequence $\{z_m\}_{m\in\mathbb{N}}$ of distinct points of $X$ that belong to $\omega(x)$, contradicting the fact that $\omega(x)$ is finite. This proves that all points of $X$ are periodic and finishes the proof.
\end{proof}

Now we prove Theorem \ref{bla} as a particular case of the theorem below. Indeed, transitivity implies that the whole space is a single chain recurrent class. This  theorem actualy proves that either there is an infinite number of chain recurrent classes, or $X$ is finite.

\begin{theorem}\label{recdoi} If a positively $n$-expansive homeomorphism, defined in a compact metric space $X$, has the shadowing property and admits a finite number of chain recurrent classes, then $X$ is finite.
\end{theorem}

\begin{proof} We suppose $f$ admits a finite number of chain recurrent classes and prove $X$ is finite. We note that it is enough to prove the theorem when $f$ is transitive. Indeed, the restriction of $f$ to each of its chain recurrent classes is transitive, positively $n$-expansive and has the shadowing property. This proves that each chain recurrent class is finite and, hence, that the chain-recurrent set is a finite union of periodic points. Lemma \ref{Lema} proves $X$ is finite. Now, let $f$ be a transitive positively $n$-expansive homeomorphism, defined in a compact metric space $X$, with the shadowing property. We claim that there exists a finite set $A\subset X$ such that $$X=\bigcup_{x\in A}W^u(x).$$ If this is not the case, we can find a sequence $\{y_i\}_{i\in\N}\subset X$ such that $y_i\in W^u(y_j)$ if, and only if, $i=j$. Let $x$ be an arbitrary point of $X$ and for each $i\in\N$ consider the sequence $(x_k^i)_{k\in\Z}$ defined by $$x_k^i=\begin{cases}
       f^k(y_i), & k<0 \\
       f^k(x), & k\geq0.
     \end{cases}$$
This sequence is a two-sided limit pseudo orbit, so the two-sided limit shadowing property with a gap, which we obtained in Theorem \ref{gap}, assures the existence of $N\in\N$ and $z_i\in X$ that two-sided limit shadows $(x_k^i)_{k\in\Z}$ with gap $k(i)\in\Z$, where $|k(i)|\leq N$. In particular, $$z_i\in W^u(y_i)\cap W^s(f^{-k(i)}(x))$$ for every $i\in\N$. Note that $z_i=z_j$ if, and only if, $i=j$. Indeed, $z_i=z_j$ implies $z_j\in W^u(y_i)$, which, in turn, implies $y_j\in W^u(y_i)$, and this just happens when $i=j$. Now, since $z_i\in W^s(f^{-k(i)}(x))$ and $|k(i)|\leq N$ for every $i\in\N$, it follows that there exists $m\in\Z$ such that $|m|\leq N$ and $W^s(f^m(x))$ is infinite. Otherwise, $(z_i)_{i\in\N}$ would be contained in the finite set $$\bigcup_{i=-N}^NW^s(f^i(x)),$$ which is not possible since $(z_i)_{i\in\N}$ is infinite. Thus, we can consider $n+1$ different points $w_i\in W^s(f^m(x))$, with $i\in\{1,\dots,n+1\}$. For each $i\in\{1,\dots,n+1\}$ let $j(i)\in\N$ be such that $$d(f^k(w_i),f^{k+m}(x))<\frac{c}{2}$$ for every $k\geq j(i)$, where $c$ is the positive $n$-expansive constant of $f$. If we let $r=\max\{j(i); \,\,\, i\in\{1,\dots,n+1\}\}$, then it follows that $$d(f^k(w_i),f^{k+m}(x))<\frac{c}{2}$$ for every $k\geq r$ and every $i\in\{1,\dots,n+1\}$. This implies that
\begin{eqnarray*}
d(f^k(w_i),f^k(w_j))&\leq& d(f^k(w_i),f^{k+m}(x)) + d(f^{k+m}(x),f^k(w_j))\\
&\leq& \frac{c}{2}+\frac{c}{2}\\
&=&c
\end{eqnarray*}
for every $k\geq r$ and each pair $(i,j)\in\{1,\dots,n+1\}\times\{1,\dots,n+1\}$. This means that all the points $f^r(w_i)$, with $i\in\{1,\dots,n+1\}$, belong to the same local stable set of size $c$. But this contradicts the fact that $c$ is a positive $n$-expansive constant of $f$ and proves the claim. We observe that this implies that $f$ admits only a finite number of periodic points, since different periodic points belong to different unstable sets. Since $f$ is positively $n$-expansive and has the shadowing property, Proposition \ref{2} assures that the chain-recurrent set is a finite union of periodic points. Then Lemma \ref{Lema} proves that $X$ is finite.

\end{proof}

\section{The L-shadowing property}\label{localtsls}

In this section, we discuss further the L-shadowing property, defined in the introduction. We begin proving that both the shadowing property and a finite number of chain recurrent classes are necessary conditions for a homeomorphism to admit the L-shadowing property.

\begin{proposition}\label{3}
If a homeomorphism, defined in a compact metric space, has the L-shadowing property then it has the shadowing property and admits a finite number of chain recurrent classes.
\end{proposition}

\begin{proof}
To prove the shadowing property, let $\eps>0$ be given and consider $\delta>0$ such that every $\delta$-pseudo orbit, that is also a two-sided limit pseudo-orbit, is both $\eps$-shadowed and two-sided limit shadowed. It is known that the shadowing property is equivalent to the finite shadowing property (see \cite{CDX} for a proof). Then it is sufficient to prove that any finite $\delta$-pseudo orbit is $\eps$-shadowed. We consider a finite $\delta$-pseudo orbit $(x_k)_{k=0}^l$ and extend it to a $\delta$-pseudo orbit, that is also a two-sided limit pseudo orbit, considering the past orbit of $x_0$ and the future orbit of $x_l$. More precisely, define the sequence $(y_k)_{k\in\Z}$ as follows
$$y_k=\begin{cases} f^k(x_0), & k<0\\
x_k, & 0\leq k\leq l \\
f^k(x_l), & k>l.
\end{cases}$$ Then the L-shadowing property assures the existence of $z\in X$ that both $\eps$-shadows and two-sided limit shadows $(y_k)_{k\in\Z}$. In particular, $z$ $\eps$-shadows the finite pseudo orbit $(x_k)_{k=0}^l$. Since this can be done for every $\eps>0$, the shadowing property is proved. To prove that $f$ admits a finite number of chain recurtrent classes we let $\alpha$ be the diameter of the space and consider $\delta>0$, given by the L-shadowing property, such that  every $\delta$-pseudo orbit, that is also a two-sided limit pseudo-orbit, is both $\alpha$-shadowed and two-sided limit shadowed. Consider a finite cover of $X$ by open sets of the form $B(x_i,\frac{\delta}{2})$, $x_i\in X$ and $i\in\{1,\dots,l\}$. We claim that two different chain recurrent classes cannot intersect the same open set of this cover. Indeed, if $y_1$ and $y_2$ belong to two different chain recurrent classes and also belong to the same open set $B(x_i,\frac{\delta}{2})$ then we can consider a $\delta$-pseudo orbit $(x_k)_{k\in\Z}$, that is also a two-sided limit pseudo orbit, defined as follows: 
$$x_k=\begin{cases}f^k(y_1), & k<0\\
f^k(y_2), & k\geq0. 
\end{cases}$$ The L-shadowing property assures the existence of $z\in X$ that both $\alpha$-shadows and two-sided limit shadows $(x_k)_{k\in\Z}$. For each $\eps>0$ we will construct an $\eps$-pseudo orbit connecting $y_1$ to $y_2$. Since $z$ two-sided limit shadows $(x_k)_{k\in\Z}$, we can choose natural numbers $k_1$ and $k_2$ satisfying $$d(f^{-k_1}(z),f^{-k_1}(y_1))<\eps \,\,\,\,\,\, \,\,\,\,\,\, \textrm{and} \,\,\,\,\,\, \,\,\,\,\,\, d(f^{k_2}(z),f^{k_2}(y_2))<\eps.$$ Chain recurrent classes are invariant, so $y_1$ and $f^{-k_1}(y_1)$ belong to the same chain recurrent class. Then there exists a finite $\eps$-pseudo orbit connecting $y_1$ and $f^{-k_1}(y_1)$. Analogously, $y_2$ and $f^{k_2}(y_2)$ belong to the same chain-recurrent class and there exists an $\eps$-pseudo orbit from $f^{k_2}(y_2)$ to $y_2$. Thus, we can concatenate the pseudo orbit connecting $y_1$ to $f^{-k_1}(y_1)$, with the segment of orbit from $f^{-k_1}(z)$ to $f^{k_2}(z)$ and the pseudo orbit from $f^{k_2}(y_2)$ to $y_2$ to create an $\eps$-pseudo orbit connecting $y_1$ and $y_2$. This can be done for every $\eps>0$. In a similar way, we consider the $\delta$-pseudo orbit, that is also a two-sided limit pseudo orbit, defined as the past orbit of $y_2$ and the future orbit of $y_1$. The L-shadowing property assures this sequence is two--sided limit shadowed and we repeat the argument above to create, for each $\eps>0$, an $\eps$-pseudo orbit connecting $y_2$ to $y_1$. This contradicts the fact that $y_1$ and $y_2$ belong to different chain recurrent classes and proves the claim. The proposition easily follows from this claim, since any open set of the form $B(x_i,\frac{\delta}{2})$ intersects at most one chain-recurrent class.
\end{proof}

Now the L-shadowing property is obtained for topologically hyperbolic homeomorphisms.

\begin{proposition}\label{expansive}
If an expansive homeomorphism, defined in a compact metric space, has the shadowing property, then it has the L-shadowing property.
\end{proposition}

\begin{proof}Consider a number $\eps>0$, given by expansiveness, such that $W^s_{\eps}(x)\subset W^s(x)$ and $W^u_{\eps}(x)\subset W^u(x)$ for every $x\in X$. Let $\delta>0$, given by the shadowing property, be such that every $\delta$-pseudo orbit is $\frac{\eps}{2}$-shadowed, and let $(x_k)_{k\in\Z}$ be a $\delta$-pseudo orbit that is also a two-sided limit pseudo-orbit. It is known that $f$ and $f^{-1}$ have the limit shadowing property, so there exist $p_1,p_2\in X$ that limit shadows $(x_k)_{k\in\Z}$ in the past and in the future, respectively. Let $k_0\in\N$ be such that $d(f^k(p_1),x_k)<\frac{\eps}{2}$ for every $k\leq -k_0$ and $d(f^k(p_2),x_k)<\frac{\eps}{2}$ for every $k\geq k_0$. Consider a point $z\in X$ that $\frac{\eps}{2}$-shadows $(x_k)_{k\in\Z}$ and note that if $k\leq -k_0$ then
\begin{eqnarray*}
d(f^k(z),f^k(p_1))&\leq&d(f^k(z),x_k)+d(x_k,f^k(p_1)\\
&<&\frac{\eps}{2}+\frac{\eps}{2}\\
&=&\eps,
\end{eqnarray*}
and if $k\geq k_0$ then
\begin{eqnarray*}
d(f^k(z),f^k(p_2))&\leq&d(f^k(z),x_k)+d(x_k,f^k(p_2)\\
&<&\frac{\eps}{2}+\frac{\eps}{2}\\
&=&\eps.
\end{eqnarray*}
Hence, $f^{-k_0}(p_1)\in W^u_{\eps}(f^{-k_0}(z))$ and $f^{k_0}(p_2)\in W^s_{\eps}(f^{k_0}(z))$. The choice of $\eps$ assures that $$d(f^k(p_1),f^k(z))\to0, \,\,\,\,\,\, k\to-\infty,$$ and that $$d(f^k(p_2),f^k(z))\to0, \,\,\,\,\,\, k\to\infty.$$ Since $p_1$ limit shadows $(x_k)_{k\in\Z}$ in the past and $p_2$ limit shadows $(x_k)_{k\in\Z}$ in the future, we obtain that $z$ two-sided limit shadows $(x_k)_{k\in\Z}$.
\end{proof}

\vspace{+0.4cm}

In this proof, the existence of the number $\eps>0$ satisfying $W^s_{\eps}(x)\subset W^s(x)$ and $W^u_{\eps}(x)\subset W^u(x)$ for every $x\in X$ is essential. In the $n$-expansive scenario, the number $n(f^{k_0}(z),\eps)$ could be greater than one and the points $f^{k_0}(z)$ and $f^{k_0}(p_2)$ could belong to different stable sets. In this case, $z$ would not limit shadow $(x_k)_{k\in\N}$. Now we prove Theorem \ref{h}, where the L-shadowing property is proved to be equivalent to the space being finite, for positively $n$-expansive homeomorphisms.

\vspace{+0.4cm}

\hspace{-0.45cm}\emph{Proof of Theorem \ref{h}} : We will use the L-shadowing property to obtain a finite set $A\subset X$ such that $$X=\bigcup_{x\in A}W^u(x).$$ Since distinct periodic points have distincts unstable sets, this would prove that the set of periodic points of $f$ is finite. Since $f$ is positively $n$-expansive and has the shadowing property, Proposition \ref{2} would assure that $\Omega(f)$ is finite and Lemma \ref{Lema} assures this is equivalent to $X$ being finite.


We claim that there exists $0<\delta<\frac{c}{2}$ such that for each $x\in X$ there exists a subset $A(x)\subset B(x,\delta)$ with at most $n$ different points such that points in $A(x)$ belong to different unstable sets and also any point in $B(x,\delta)$ belongs to the unstable set of some point in $A(x)$. Let $c>0$ be the positively $n$-expansive constant of $f$ and $\delta>0$, given by the L-shadowing property, be such that every $\delta$-pseudo orbit, that is also a two-sided limit pseudo-orbit, is both $\frac{c}{2}$-shadowed and two-sided limit shadowed. If for some $x\in X$ there is not such set $A(x)$ in $B(x,\delta)$, then there exists a subset $C\subset B(x,\delta)$ with at least $n+1$ different points such that points in $C$ belong to different unstable sets. For each $y\in C$ we consider the sequence $\{x_k\}_{k\in\mathbb{Z}}\subset X$ defined by $$x_k=\begin{cases}
       f^k(y), & k<0 \\
       f^k(x), & k\geq0.
     \end{cases}$$
This sequence is a $\delta$-pseudo orbit, that is also a two-sided limit pseudo orbit, then the L-shadowing property assures the existence of $z(y)$ that both $\frac{c}{2}$-shadows and two-sided limit shadows $(x_k)_{k\in\Z}$. In particular, $$z(y)\in W^u(y)\cap W^u_{\frac{c}{2}}(y)\cap W^s_{\frac{c}{2}}(x)\cap  W^s(x).$$ Distincts points in $B$ have distinct unstable sets, then, $y,w\in B$ and $y\neq w$ imples $z(y)\neq z(w)$. Thus, the set $\{z(y); \,\, y\in B\}$ has at least $n+1$ points and is contained in $W^s_{c}(x)$, contradicting the fact that $f$ is positively $n$-expansive. This proves the claim. To finish the proof of the theorem, we choose a finite cover of $X$ by open sets of the form $B(x_i,\delta)$ with $x_i\in X$ and $i\in\{1,\dots,l\}$ and consider the sets $A(x_i)\subset B(x_i,\delta)$ given by the previous claim. It is easy to note that the union $$A=\bigcup_{i=1}^lA(x_i)$$ satisfies $$X=\bigcup_{x\in A}W^u(x).$$  As observed above, this proves that $X$ is finite.\qed

\vspace{+0.7cm}

\hspace{-0.45cm}\emph{Proof of Theorem \ref{D}} : Let $c>0$ be the $n$-expansivity constant of $f$ and let $\delta>0$, given by the L-shadowing property, be such that every $\delta$-pseudo orbit, that is also a two-sided limit pseudo-orbit, is both $\frac{c}{2}$-shadowed and two-sided limit shadowed. Finally, let $\eps=\frac{1}{2}\min\{c,\delta\}$ and consider the sets $E(x,\eps)$ and $\bar{E}(x,\eps)$ as in the definition of $n(x,\eps)$ and $\bar{n}(x,\eps)$, respectively. For any pair $(q,p)\in\bar{E}(x,\eps)\times E(x,\eps)$, the sequence $(x_k)_{k\in\Z}$ defined by $$x_k=\begin{cases}
       f^k(q), & k<0 \\
       f^k(p), & k\geq0\end{cases}$$ is both a $\delta$-pseudo orbit and a two-sided limit pseudo orbit. Thus, there exists $z\in X$, given by the L-shadowing property, that $\frac{c}{2}$-shadows and also two-sided limit shadows $(x_k)_{k\in\Z}$. This defines a map $z:\bar{E}(x,\eps)\times E(x,\eps)\to X$ such that $$z(q,p)\in W^u_{\frac{c}{2}}(q)\cap W^s_{\frac{c}{2}}(p)\cap W^u(q)\cap W^s(p).$$ This map $z$ is injective, since $(q,p)\neq(r,s)$ implies $z(q,p)$ and $z(r,s)$ belong to different stable (or unstable) sets, so $z(q,p)\neq z(r,s)$. Let $Rg(z)$ be the range of the map $z$. Since $Rg(z)\subset W^u_c(x)\cap W^s_c(x)$ and $\#Rg(z)=n(x,\eps).\bar{n}(x,\eps)$, the desired inequality follows from the fact that $c$ is an $n$-expansivity constant of $f$. 
			
If we assume, in addition, that $f$ is transitive, a similar argument obtains the desired inequality for any pair of points $(x,y)\in X\times X$. In this case, $\delta>0$ will be such that every $\delta$-pseudo orbit, that is also a two-sided limit pseudo-orbit, is both $\frac{c}{3}$-shadowed and two-sided limit shadowed. If $\eps=\frac{1}{3}\min\{c,\delta\}$ and $w\in X$ satisfies $$d(w,\bar{E}(x,\eps))<\delta \,\,\,\,\,\, \textrm{and} \,\,\,\,\,\,d(f^{m}(w),E(y,\eps))<\delta$$ for some $m>0$, then for any pair $(q,p)\in\bar{E}(x,\eps)\times E(y,\eps)$ the sequence $(x_k)_{k\in\Z}$ defined by $$x_k=\begin{cases}
       f^k(q), & k<0 \\
       f^k(w), & 0\leq k<m \\ 
			 f^{k-m}(p), & k\geq m\end{cases}$$ is both a $\delta$-pseudo orbit and a two-sided limit pseudo orbit. The shadowing point $z(q,p)\in X$, given by the L-shadowing property, satisfies $$z(q,p)\in W^u_{\frac{c}{3}}(q)\cap W^u(q) \,\,\, \text{ and } \,\,\, f^{m}(z(q,p))\in W^s_{\frac{c}{3}}(p)\cap W^s(p).$$ The map $z$ is injective, since $(q,p)\neq(r,s)$ implies $z(q,p)$ and $z(r,s)$ belong to different unstable sets, or $f^m(z(q,p))$ and $f^m(z(r,s))$ belong to different stable sets. Also, points in $Rg(z)$ belong to the same dynamical ball of radius $c$, since they all $\frac{c}{3}$-shadow pseudo orbits that are $\frac{c}{3}$-close. As before, the desired inequality follows from the fact that $c$ is an $n$-expansivity constant of $f$. \qed




\section{Acknowledgments}

Part of this work was developed while the first author was working at UFV-Brazil and the second author was visiting this institutions' mathematics department. The second author was supported by CNPq(Brazil).

\vspace{1cm}
\noindent

{\em B. Carvalho}
\vspace{0.2cm}

\noindent

Departamento de Matem\'atica,

Universidade Federal de Minas Gerais - UFMG

Av. Ant\^onio Carlos, 6627 - Campus Pampulha

Caixa Postal: 702

31270-901, Belo Horizonte - MG, Brazil.
\vspace{0.2cm}

\email{bmcarvalho@mat.ufmg.br}

\vspace{1.5cm}
\noindent

{\em W. Cordeiro}

\noindent

IMPA

Estrada Dona Castorina 110

Rio de Janeiro / Brasil

22460-320

\vspace{0.2cm}

\email{welingtonscordeiro@gmail.com}

\end{document}